\documentclass{amsart}

\usepackage{bbm}
\usepackage{amsmath}
\usepackage{amssymb}
\usepackage[all]{xy}
\usepackage{enumerate}

\newtheorem{theorem}{Theorem}[section] 
\newtheorem{lemma}[theorem]{Lemma}     
\newtheorem{corollary}[theorem]{Corollary}
\newtheorem{proposition}[theorem]{Proposition}

\theoremstyle{definition}
\newtheorem{definition}{Definition}

\DeclareMathOperator{\id}{id}

\providecommand{\setC}{\mathbb{C}}

\providecommand{\setR}{\mathbb{R}}
\providecommand{\setS}{\mathbb{S}}
\providecommand{\setU}{\mathbb{U}}
\providecommand{\setZ}{\mathbb{Z}}

\providecommand{\tensor}{\otimes}
\providecommand{\eps}{\varepsilon}

\providecommand{\incc}[1]{\left[{#1 }\right]}
\providecommand{\calA}{\mathcal{A}}
\providecommand{\calI}{\mathcal{I}}

\providecommand{\calO}{\mathcal{O}}
\providecommand{\conj}[1]{\overline{#1 }}
\providecommand{\quot}[2]{{#1}/{#2}}

\providecommand{\slra}[1]{\stackrel{#1}{\longrightarrow}}
\providecommand{\sfoot}[1]{\mbox{\begin{footnotesize}$#1$\end{footnotesize}}}

\providecommand{\stiny}[1]{\mbox{\begin{tiny}$#1$\end{tiny}}}
\providecommand{\trace}[1]{\mbox{Tr}\,{#1}}

\title[Equivariant Homotopy Invariants]{Axiomatic Description of Lefschetz Type Equivariant Homotopy Invariants}
\author{Philipp Wruck}

\begin{document}

\begin{abstract}
We describe generators of universal Lefschetz groups consisting of self-maps of equivariant 1-spheres. This allows to formulate a normalization axiom which, together with the usual axioms,
determines an equivariant Lefschetz number uniquely. We apply this technique to identify the fixed orbit detecting equivariant Lefschetz number of Chorny with the fixed orbit 
index of Dzedzej.
\end{abstract}

\maketitle

\section{Introduction}
 The Lefschetz number is a homotopy invariant of a self-map $f:X\to X$ of a finite CW complex intimately related to the fixed points of the map $f$. Arkowitz and Brown have shown in 
 \cite{arkowitz} that the Lefschetz number is uniquely characterized by four axioms, which are quite natural to postulate.

 Things are a bit trickier in the equivariant world. Laitinen and L\"{u}ck \cite{laitinen} have constructed a universal Lefschetz group and a universal Lefschetz number, where 
 universality means that every equivariant homotopy invariant with certain properties is the homomorphic image of the universal invariant. There are still several useful instances of 
 equivariant Lefschetz numbers, and it is desirable to be able to tell when two such invariants are equal, i.e. to determine the homomorphism which specifies them. 
 
 In this paper we are going to give a characterization of the universal Lefschetz group for any topological group $G$ in terms of generators. The generators consist of equivariant 
 homotopy classes of self maps of finite wedge sums of equivariant $1$-spheres. By introducing a linearity axiom in the definition of an equivariant Lefschetz number, we obtain a 
 universal linear Lefschetz group and a universal linear equivariant Lefschetz number. This is the Lefschetz invariant of Laitinen and L\"{u}ck. For the universal linear Lefschetz group, 
 we obtain a simpler set of generators, which turns out to be a basis under mild hypotheses on $G$. This last result was already imminent in \cite{laitinen}. 

 Formulated in the Axiomatic language in the spirit of Arkowitz and Brown, the characterization implies that any equivariant Lefschetz number is uniquely determined by its values on self 
 maps of finite wedge sums of equivariant 1-spheres. This extends a result of Gon\c{c}alves and Weber \cite{goncalves}, who have identified yet another equivariant Lefschetz number of 
 L\"{u}ck and Rosenberg \cite{rosenberg} for cocompact actions as the unique equivariant Lefschetz invariant having as values on finite wedge sums of equivariant 1-spheres certain linear 
 combinations of incidence numbers. 

 In Section 2, we develop the universal Lefschetz group and the universal linear Lefschetz group, construct the universal invariants, and describe generators and bases of the universal 
 groups. The main result is Theorem \ref{thm:lefschetzbasis}. In Corollary \ref{cor:lefschetzaxioms}, we reformulate the result as a set of axioms for equivariant Lefschetz numbers.

 In Section 3, we develop the theory of equivariant cellular homology to be able to define an equivariant Lefschetz number later. This chapter is based on \cite{ynrohc} and \cite{willson}
 and nothing new is proven here. At some points the representation may be a bit more simple than in the references, but on the other hand, we have also skipped some of the proofs.

 Section 4 reviews the theory of the Hattori-Stallings trace of endomorphisms of finitely generated projective modules and proves some of the more unknown properties of this map.
 
 In Section 5 and Section 6, we use the axioms to compare two constructions of an equivariant Lefschetz number which, in contrast to L\"{u}ck and Rosenbergs invariant, detects fixed orbits
 of $G$-maps rather than fixed points. These constructions are a homological construction, extracted from \cite{ynrohc}, and an analytical construction done by Dzedzej in \cite{dzedzej}. 
 It turns out that these two invariants are equal whenever both are defined. In many instances, the analytical Lefschetz number is much easier to calculate than the homological one.

 We close in Section 7 with the obligatory Lefschetz fixed orbit theorem and an example that shows that the formulation is, in some sense, optimal.

\section{The Universal Lefschetz Number}
 In the following we will work in the category of pointed $G$-CW complexes, where $G$ is any topological group. $G$ will be more specific later on. A quick definition of the notion of 
 $G$-CW complexes is given at the beginning of Section \ref{section:cellhom}. We will not refer to the fact that complexes and maps are pointed. So by a $G$-map $f:X\to X$ of a $G$-CW 
 complex, we really mean a pointed $G$-map of a pointed $G$-CW complex with $G$-fixed base point. Similarly, equivariant homotopies are required to fix base points.

\begin{definition}
 An equivariant Lefschetz number is a pair $(\ell_G, B)$, where $B$ is an abelian group and $\ell_G$ is an assignment, assigning to any equivariant self map $f$ of a finite $G$-CW 
 complex $X$ an element $\ell_G(f)$ of $B$. This is required to fulfill:
\begin{enumerate}
 \item (Homotopy Invariance): If $k:X\to Y$ is a $G$-homotopy equivalence of finite $G$-CW complexes, $f:X\to X$, $h:Y\to Y$ and the diagram
 \[
  \xymatrix{
           X\ar[r]^f\ar[d]^k&X\ar[d]^k\\
           Y\ar[r]^h&Y
           }
 \]
 is $G$-homotopy commutative, then $\ell_G(f)=\ell_G(h)$.

 \item (Cofibration): If $A\subseteq X$ is an invariant subcomplex and there is a diagram
  \[
   \xymatrix{
             A\ar[r]\ar[d]_{f|_A}&X\ar[d]^f\ar[r]&\quot XA\ar[d]^{\hat{f}}\\
             A\ar[r]&X\ar[r]&\quot XA,
            }
  \]
 then
 \[
  \ell_G(f)=\ell_G(f\big|_A)+\ell_G(\hat{f}).
 \]

 \item (Commutativity): If $X, Y$ are $G$-CW complexes and $f:X\to Y$, $h:Y\to X$ are $G$-maps, then
 \[
  \ell_G(h\circ f)=\ell_G(f\circ h).
 \]
\end{enumerate}
\end{definition}
   
A universal equivariant Lefschetz number is a Lefschetz number $(L_G, U)$ such that, whenever $(\ell_G, B)$ is an equivariant Lefschetz number, there is a unique homomorphism 
$\varphi:U\to B$ such that $\ell_G(f)=\varphi(L_G(f))$ for any self map $f:X\to X$ of a finite $G$-CW complex.

The following purely formal result shows that a universal equivariant Lefschetz number exists.

\begin{proposition}\label{prop:existslefschetz}
 There exists a universal Lefschetz number $(L_G, \setU_GL)$.
\end{proposition}

\begin{proof}
 Consider the set of pairs $[f, X]$ of equivalence classes of pairs $(f, X)$, where $(f, X)$ is equivalent to $(h, Y)$, if there is a $G$-homotopy equivalence $k:X\to Y$, making the 
 diagram
 \[
  \xymatrix{
           X\ar[r]^f\ar[d]^k&X\ar[d]^k\\
           Y\ar[r]^h&Y
           }
 \]
 $G$-homotopy commutative. Let $U$ be the free abelian group on these equivalence classes and $S$ be the subgroup generated by elements of one of the following forms.
 \begin{enumerate}
  \item 
  \[
  [f, X]-[f\big|_A, A]-[\hat{f}, \quot XA],
  \]
  where $A\subseteq X$ is an invariant subcomplex, $f(A)\subseteq A$ and $\hat{f}$ is the map induced in the quotient.

  \item
  \[
   [f\circ h, X]-[h\circ f, Y],
  \]
  where $h:X\to Y$ and $f:Y\to X$ are $G$-maps. 
 \end{enumerate}
 Denote the quotient group $\quot US$ by $\setU_GL$. For a self-map $f:X\to X$, let $L_G(f)\in\setU_GL$ be the class of $[f, X]$. We claim that this is a universal Lefschetz invariant. 
 Indeed, if $(\ell_G, B)$ is any Lefschetz invariant, we can define a map $\varphi:\setU_GL\to B$ by $[f, X]\mapsto\ell_G(f)$ and extend linearly. This is well-defined because $\ell_G$ is
 an equivariant Lefschetz number and obviously, $\varphi$ is a homomorphism with $\varphi(L_G(f))=\ell_G(f)$.
\end{proof}

Before we describe the universal Lefschetz group more explicitly, we derive some properties of equivariant Lefschetz numbers in general.

\begin{proposition}\label{prop:lefschetzproperties}
 Let $(\ell_G, B)$ be an equivariant Lefschetz number. Then the following holds.
  \begin{enumerate}[(i)]
   \item For the one point space $*$ and the unique map $*:*\to *$, $\ell_G(*)=0$.

   \item If $f:X\to X$ is equivariantly homotopic to the constant map, then $\ell_G(f)=0$.

   \item For a $G$-CW complex $X$, its ordinary suspension $\Sigma X=X\wedge\setS^1$ and $f:X\to X$,
   \[
    \ell_G(f)=-\ell_G(\Sigma f).
   \]
  \end{enumerate}
\end{proposition}

\begin{proof}
 \begin{enumerate}[(i)]
  \item Clearly for the sequence $*\slra{*}*\slra{*}*$, the cofibration property of the equivariant Lefschetz number implies
  \[
   \ell_G(*)=\ell_G(*)+\ell_G(*)
  \]
  and consequently $\ell_G(*)=0$.

  \item Let $c:X\to X$ be the constant map. $c$ factors as $c=i\circ p$, where $p:X\to *$ is the projection, $i:*\to X$ the inclusion of the base point. By the commutativity property and 
  (i),
  \[
   \ell_G(c)=\ell_G(i\circ p)=\ell_G(p\circ i)=\ell_G(*)=0.
  \]
 
  \item Let $CX$ be the reduced cone over $X$, then $\Sigma X\cong\quot{CX}X$ and the map $Cf:CX\to CX$ preserves $X\subseteq CX$. Furthermore, the map $Cf$ induces on the quotient is 
  just $\Sigma f$. Hence we obtain
  \[
   \ell_G(Cf)=\ell_G(f)+\ell_G(\Sigma f).
  \]
  Since $CX$ is $G$-contractible, $Cf$ is equivariantly homotopic to the constant map and therefore $\ell_G(Cf)=0$, which yields the result.
 \end{enumerate}
\end{proof}

The next lemma is crucial for several inductive proofs on the dimension of CW complexes.

\begin{lemma}\label{lem:spheresuspension}
 Let $X=\bigvee{\quot G{H_i}}_+\wedge\setS^{n-1}$ be a finite wedge of equivariant $(n-1)$-spheres, $f:X\to X$ a $G$-map. Then $f$ is equivariantly homotopic to the $(n-2)$-nd ordinary 
 suspension of a $G$-map of the complex
\[
 Y=\bigvee{\quot G{H_i}}_+\wedge\setS^1,
\]
provided $n\geq2$.
\end{lemma}

\begin{proof}
 We want to show that the suspension homomorphism
 \[
  \Sigma:[{\quot GK}_+\wedge\setS^{n-1}, X]_G\to[{\quot GK}_+\wedge\setS^n,\Sigma X]_G
 \]
 is an epimorphism under the condition that $X$ is a wedge sum of equivariant $(n-1)$-spheres. We have the usual adjunction isomorphism
 \[
  [{\quot GK}_+\wedge\setS^{n-1}, X]_G\cong[\setS^{n-1}, X^K].
 \]
 Hence, it suffices to show that 
 \[
  \Sigma:[\setS^{n-1}, X^K]\to[\setS^n,\Sigma X^K]
 \]
 is an epimorphism. By the Freudenthal suspension theorem, if $X^K$ is $m$-connected, then
\[
 \Sigma:\pi_r(X^K)\to\pi_{r+1}(\Sigma X^K)
\]
 is an epimorphism for $r\leq 2m+1$. We have
 \begin{eqnarray*}
  X^K&=&\{([g], x)\;|\;([kg], x)=([g], x)\;\forall\,k\in K\}\cup\{*\}\\
     &=&\{([g], x)\;|\;([g], x)\in\quot G{H_i}\times\setS^{n-1},\;g^{-1}Kg\subseteq H_i\}\cup\{*\}\\
     &=&\left(\bigvee_{i\in I}N(K, H_i)_+\right)\wedge\setS^{n-1},
 \end{eqnarray*}
 where $N(K, H)=\{g\in G\;|\;g^{-1}Kg\subseteq H\}$. Since $\setS^{n-1}$ is $(n-2)$-connected and $X^K$ is the smash product of $\setS^{n-1}$ with a $-1$-connected space,
 $X^K$ is itself $(n-2)$-connected. So the Freudenthal theorem yields that suspension is an epimorphism, provided $r\leq 2(n-2)+1=2n-3$. Letting $r=n-1$, the condition reduces to 
 $0\leq n-2$, i.e. $n\geq2$.
\end{proof}

In the present generality, equivariant Lefschetz numbers are not well-suited for applications. For example, there is in general no way to give a connection between an equivariant 
Lefschetz number and the fixed points or orbits of a map. Therefore it is customary to assume further axioms. By universality, the resulting invariants will be homomorphic images of the 
universal Lefschetz number. 

Laitinen and L\"{u}ck in \cite{laitinen} have introduced the following Axiom.
\begin{enumerate}
 \item[4a.] (Linearity): If $X$ is a $G$-CW complex and $f, h:\Sigma X\to\Sigma X$, then
 \[
  \ell_G(f+h)=\ell_G(f)+\ell_G(h),
 \]
 where $f+h$ is a representant of the sum of the homotopy classes of $f$ and $h$, coming from the structure of $\Sigma X$ as an $H$-cogroup.
\end{enumerate}

Arkowitz and Brown in \cite{arkowitz} gave an additivity Axiom, though without an acting group $G$. An equivariant version appears in \cite{goncalves}. It is apparent that additivity must 
hold for any Lefschetz number that is related to fixed points or orbits of a map. 

\begin{enumerate}
 \item[4b.] (Wedge of Circles): Whenever $X=\bigvee{\quot G{H_i}}_+\wedge\setS^1$ is a finite wedge of equivariant $1$-spheres and $f:X\to X$ a $G$-map, then
 \[
  \ell_G(f)=\sum\ell_G(f_i),
 \]
 where $f_i=\pi_i\circ f\circ\iota_i$, with $\iota_i:{\quot G{H_i}}_+\wedge\setS^1\to X$ the inclusion of the $i$-th sphere, $\pi_i:X\to{\quot G{H_i}}_+\wedge\setS^1$ the projection onto 
 it.
\end{enumerate}

The wedge of circles axiom ensures that an equivariant Lefschetz number is additive in general.

\begin{proposition}\label{prop:wedgeadditivity}
 Let $(\ell_G, B)$ be an equivariant Lefschetz number satisfying the wedge of circles axiom and let $X=\bigvee X_i$ be a finite wedge of $G$-CW complexes, $f:X\to X$ a $G$-map. Then
 \[
  \ell_G(f)=\sum\ell_G(f_i),
 \]
 where $f_i$ is the obvious adjustment of the map $f_i$ from Axiom 4b.
\end{proposition}

\begin{proof}
 We begin by proving that the assertion holds for wedge sums of equivariant $n$-spheres. Hence, let 
 \[
  X=\bigvee{\quot G{H_i}}_+\wedge\setS^n,\;f:X\to X.
 \]
 By Lemma \ref{lem:spheresuspension}, $f$ is equivariantly homotopic to the repeated suspension of a $G$-map $h:Y\to Y$, where
 \[
  Y=\bigvee{\quot G{H_i}}_+\wedge\setS^1.
 \]
 By the wedge of circles Axiom, $\ell_G(h)=\sum\ell_G(h_i)$. By homotopy invariance and the obvious equality $\pi_i\circ\Sigma h\circ\iota_i=\Sigma(\pi_i\circ h\circ\iota_i)$, it follows 
 that
 \begin{eqnarray*}
  \ell_G(f_i)&=&\ell_G(\pi_i\circ f\circ\iota_i)\\
             &=&\ell_G(\pi_i\circ\Sigma^{n-1}h\circ\iota_i)\\
             &=&\ell_G(\Sigma^{n-1}h_i)\\
             &=&(-1)^{n-1}\ell_G(h_i).
 \end{eqnarray*}
Therefore,
\begin{eqnarray*}
 \ell_G(f)&=&\ell_G(\Sigma^{n-1}h)\\
          &=&(-1)^{n-1}\ell_G(h)\\
          &=&(-1)^{n-1}\sum\ell_G(h_i)\\
          &=&(-1)^{n-1}\sum(-1)^{n-1}\ell_G(f_i)\\
          &=&\sum\ell_G(f_i).
\end{eqnarray*}
The general case now follows easily by induction on the dimension of the wedge summands. 

If $X$ is a wedge sum of orbits, the suspension of $X$ is a wedge of equivariant 1-spheres. By the wedge of circles axiom and Proposition \ref{prop:lefschetzproperties} (iii), the claim 
holds for $X$. 

For the induction step, let all wedge summands of $X$ be at most $n$-dimensional. By the cofibration Axiom,
\[
 \ell_G(f)=\ell_G(f_{n-1})+\ell_G(\hat{f}),
\]
where $f_{n-1}$ is the restriction of $f$ to the $n-1$-skeleton, where we may assume that $f$ is cellular, and $\hat{f}$ is the map induced on $\quot{X_n}{X_{n-1}}$. By induction, the 
claim holds for $f_{n-1}$ and it holds for $\hat{f}$ by what we have proven already. Hence, it holds for $f$ itself.
\end{proof}

Next we prove that Axioms 4a and 4b are equivalent.

\begin{proposition}\label{prop:axiomslinwedge}
 An equivariant Lefschetz number satisfies the linearity axiom if and only if it satisfies the wedge of circles Axiom.
\end{proposition}
\begin{proof}
 We prove that linearity is equivalent to additivity in the sense of Proposition \ref{prop:wedgeadditivity}. By means of that Proposition, the result will follow.
 
 We assume first that $\ell_G$ is linear. By induction, it then suffices to assume that $f:X\to X$ is a $G$-map of a complex consisting of two wedge summands, $X=A\vee B$. 

 Let $p_A:X\to X$ be the projection onto $A$, similarly $p_B$. Then $\Sigma f$ is equivariantly homotopic to $\Sigma(f\circ p_A)+\Sigma(f\circ p_B)$. By Proposition 
 \ref{prop:lefschetzproperties} (iii), commutativity and the facts $p_A^2=p_A$, $\Sigma p_A=p_{\Sigma A}$ and similarly for $B$, we obtain
 \begin{eqnarray*}
  \ell_G(f)&=&-\ell_G(\Sigma f)\\
           &=&-\ell_G(\Sigma(f\circ p_A))-\ell_G(\Sigma(f\circ p_B))\\
           &=&-\ell_G(\Sigma f\circ p_{\Sigma A})-\ell_G(\Sigma f\circ p_{\Sigma B})\\
           &=&-\ell_G(\Sigma f\circ p^2_{\Sigma A})-\ell_G(\Sigma f\circ p^2_{\Sigma B})\\
           &=&-\ell_G(p_{\Sigma A}\circ\Sigma f\circ p_{\Sigma A})-\ell_G(p_{\Sigma B}\circ\Sigma f\circ p_{\Sigma B}).
 \end{eqnarray*}
The map $p_{\Sigma A}\circ\Sigma f\circ p_{\Sigma A}$ leaves $\Sigma_A$ invariant and is the constant map on $\Sigma B$. Cofibration shows that
\[
 \ell_G(p_{\Sigma A}\circ\Sigma f\circ p_{\Sigma A})=\ell_G(\Sigma f_A),
\]
where $f_A:A\to A$ is the map induced by $f$ via inclusion and projection. Of course a similar result holds with $A$ replaced by $B$ and consequently,
\[
 \ell_G(f)=-\ell_G(\Sigma f_A)-\ell_G(\Sigma f_B)=\ell_G(f_A)+\ell_G(f_B).
\]
For the other direction, we assume that $\ell_G$ is additive. Let $f, h:\Sigma X\to\Sigma X$ be two $G$-maps. Let $\gamma:\Sigma X\to\Sigma X\vee\Sigma X$ be the folding map. By 
additivity,
\[
 \ell_G(\gamma\circ(f\vee h))=\ell_G(p_1\circ\gamma\circ f)+\ell_G(p_2\circ\gamma\circ h),
\]
where $p_i:\Sigma X\vee\Sigma X\to\Sigma X$ is the projection to the $i$-th summand, $i=1, 2$. The maps $p_i\circ\gamma$ are equivariantly homotopic to the identity. By commutativity,
we can conclude
\begin{eqnarray*}
 \ell_G(f)+\ell_G(h)&=&\ell_G(\gamma\circ(f\vee h))\\
                    &=&\ell_G((f\vee h)\circ\gamma)\\
                    &=&\ell_G(f+h).
\end{eqnarray*}
\end{proof}

With the additional linearity Axiom, there is a universal linear Lefschetz group $U_GL$ and a universal linear equivariant Lefschetz number satisfying all four axioms. This is easily
derived in the same manner as Proposition \ref{prop:existslefschetz}, with adding the linearity condition as an additional set of relations. By universality, there is a homomorphism 
$\setU_GL\to U_GL$ such that the image of the universal equivariant Lefschetz number is the universal linear equivariant Lefschetz number. By abuse of notation, we will denote the 
universal linear equivariant Lefschetz number by $L_G$ as well. 

Under the assumption that every orbit map $\quot GH\to\quot GH$ is a homeomorphism, Laitinen and L\"{u}ck proved that the universal Lefschetz group for linear equivariant Lefschetz 
numbers is isomorphic to
\[
 \bigoplus_{(H)}\setZ Co(\pi_0(W(H))).
\]
In the sum, $(H)$ runs through the conjugacy classes of closed subgroups of $G$, $Co(K)$ for a group $K$ denotes the set of conjugacy classes of elements of $K$ and $W(H)$ is the Weyl 
group of $H$, that is, the quotient group $\quot{N(H)}H$, where $N(H)$ is the largest subgroup of $G$ such that $H\subseteq N(H)$ is normal. In particular, the result implies that any 
linear equivariant Lefschetz number is uniquely determined by its values on self maps of orbits. Even more precisely, it is determined by its values on the maps
\[
 r_w:{\quot GH}_+\to{\quot GH}_+,\;[g]\mapsto[gw],
\]
where $H$ runs through the conjugacy classes of closed subgroups of $G$ and $w$, for each $(H)$, runs through a complete set of representants of $Co(\pi_0(W(H)))$.

We will reprove this result with slightly different techniques. In particular we will prove first that the universal Lefschetz group $\setU_GL$, without the additional linearity Axiom,
is generated by representants of self-maps of finite wedge sums of equivariant 1-spheres, which already yields a useful normalization Axiom. Assuming linearity in addition,
we show that for any topological group $G$, the basis of L\"{u}ck and Laitinen at least generates the universal linear Lefschetz group, provided one replaces the Weyl group $W(H)$ by the 
space $\left(\quot GH\right)^H$. In most examples these two spaces are homeomorphic as topological monoids, for example if $G$ is compact. When this is the case, we show that the set of 
generators is in fact a basis.

\begin{theorem}\label{thm:lefschetzbasis}
 Let $A\subseteq\setU_GL$ be the set of elements $[f, X]\in\setU_GL$, where $X$ is a finite wedge sum of equivariant $1$-spheres. Let $A'$ be the set of elements 
 $[r_w, {\quot GH}_+]\in U_GL$, with $H$ running through a complete set of representants of conjugacy classes of closed subgroups of $G$ and $r_w$ is right multiplication with $w$, $w$ 
 running through a complete set of representants of elements in $Co(\pi_0(\left(\quot GH\right)^H))$. Here, conjugacy is to be understood via the $W(H)$-action on 
 $\left(\quot GH\right)^H$.

 Then the universal Lefschetz group $\setU_GL$ is generated by $A$ and the universal linear Lefschetz group $U_GL$ is generated by $A'$. If $W(H)\cong\left(\quot GH\right)^H$ for every
 closed subgroup of $G$, then $A'$ is a basis of $U_GL$. More precisely, there are equivariant Lefschetz numbers $\ell_H$ with values in $\setZ Co(\pi_0(W(H)))$ such that
\[
 U_GL\to\bigoplus_{(H)}\setZ Co(\pi_0(W(H))),\;[f, X]\mapsto\sum_{(H)}\ell_H(f)
\]
 is an isomorphism.
\end{theorem}

\begin{proof}
 We first prove that $\setU_GL$ is generated by $A$. Let $f:X\to X$ be a self map of a finite $G$-CW complex of dimension $n$. We can assume that $f$ is cellular and so
 \[
  L_G(f)=L_G(f_{n-1})+L_G(\hat{f}),
 \]
 where $\hat{f}$ is the map induced in the quotient $\quot{X_n}{X_{n-1}}\cong\bigvee_{i\in I}{\quot G{H_i}}_+\wedge\setS^n$. By Lemma \ref{lem:spheresuspension}, $\hat{f}$ is equivariantly
 homotopic to the suspension of a $G$-map
 \[
  h:\bigvee_{i\in I}{\quot G{H_i}}_+\wedge\setS^1\to\bigvee_{i\in I}{\quot G{H_i}}_+\wedge\setS^1,
 \]
 provided $n\geq1$. 
 
 Therefore, $L_G(\hat{f})=L_G(\Sigma^{n-1}h)=(-1)^{n-1}L_G(h)$. This, together with the formula $L_G(f)=L_G(f_{n-1})+L_G(\hat{f}),$ yields by induction that
 \[
  L_G(f)=L_G(f_0)+R,
 \]
 where $R$ is a linear combination of elements in $A$. The suspension of a self map of a $0$-dimensional $G$-CW complex is an element of $A$ and
 \[
  L_G(f)=-L_G(\Sigma f).
 \]
 Hence, $L_G(f)$ is a linear combination of elements of $A$, which proves the first assertion.

 Next we assume that we are dealing with the universal linear Lefschetz group $U_GL$. Obviously all the considerations above are still valid. Moreover by additivity, for a $G$-map
 \[
  f:\bigvee{\quot G{H_i}}_+\wedge\setS^1\to\bigvee{\quot G{H_i}}_+\wedge\setS^1,
 \]
 $L_G(f)=\sum L_G(f_i)$, where the $f_i$ are defined as before. This shows that $U_GL$ is already generated by $G$-self maps of equivariant 1-spheres. By the usual reduction techniques,
 \[
  [{\quot GH}_+\wedge\setS^1,{\quot GH}_+\wedge\setS^1]_G\cong\pi_1(\left(\quot GH\right)^H_+\wedge\setS^1)\cong\bigoplus_{w\in\pi_0(\quot GH^H)}\setZ.
 \]
 In particular, the maps $[[g], z]\mapsto[[gw], z]$ generate $[{\quot GH}_+\wedge\setS^1,{\quot GH}_+\wedge\setS^1]_G$, where $w$ runs through representants of elements of 
 $\pi_0(\left(\quot GH\right)^H)$. These maps are the suspensions of maps $r_w$. Linearity again yields that $U_GL$ is generated by the classes of maps $r_w$ before passing to conjugacy 
 classes. Finally by commutativity, $L_G(r_{vwv^{-1}})=L_G(r_w)$ for $v\in W(H)$, which shows that we can pass to conjugacy classes, so $A'$ indeed generates $U_GL$.

 We now assume that $W(H)\cong\left(\quot GH\right)^H$ for every closed subgroup $H$ of $G$. To show that the elements of $A'$ are linearly independent, we need equivariant Lefschetz 
 numbers for comparison purposes. Fix a closed subgroup $H$ of $G$ and let $f:X\to X$ be a $G$-self map of a finite $G$-CW complex $X$. The cellular chain complex $C_*(X^H, X^{>H})$ is a 
 $\setZ W(H)$-module, and by homotopy invariance, it is a $\setZ\pi_0(W(H))$-module. Since $X^H\setminus X^{>H}$ is a free 
 $W(H)$-space, $C_*(X^H, X^{>H})$ is a free $\setZ\pi_0(W(H))$-module. We can assume that $f$ is cellular and thus, $f^H$ induces a $\setZ\pi_0(W(H))$-module homomorphism 
 \[
  C_*(f^H):C_*(X^H, X^{>H})\to C_*(X^H, X^{>H}).
 \]
 $C_*(f^H)$ corresponds to a matrix with entries in $\setZ\pi_0(W(H))$. The sum of the diagonal entries of this matrix followed by the projection to the 
 abelianization of $\setZ\pi_0(W(H))$ has all the properties of a trace map. The abelianization of $\setZ\pi_0(W(H))$ is 
 $\setZ Co(\pi_0(W(H)))$. We obtain a trace map 
 \[
  \trace:\hom_{\setZ\pi_0(W(H))}(C_*(X^H, X^{>H}), C_*(X^H, X^{>H}))\to\setZ Co(\pi_0(W(H))).
 \]
 This is a special instance of the Hattori-Stallings trace, compare Section \ref{section:traces} and also \cite{hattori}. 

 With this trace map, we define
\[
 \ell_H(f)=\trace C_*(f^H).
\]
It is easily seen that $\ell_H$ is a linear equivariant Lefschetz number. Hence, there is a homomorphism 
\[
 U_GL\to\setZ Co(\pi_0(W(H))),\;[f, X]\mapsto\ell_H(f). 
\]
We turn back to the question of linear independence of $A'$. Let an equation
 \[
  \sum_{i\in I}\lambda_ia_i=0
 \]
 be given, where the $a_i$ are pairwise distinct elements of $A'$. Let $K$ be a closed subgroup of $G$ with the property that there is an index $i\in I$ with $\lambda_i\neq0$ and $K$ is 
 maximal with this property with respect to the partial order on isotropy types $(H)$ induced by subconjugacy. Since 
 \[
  \left({\quot GH}_+\right)^K=*
 \]
for subgroups $H$ with either $(H)<(K)$ or with $(H)$ and $(K)$ incomparable, $\ell_K(a_i)=0$ for elements $a_i$ which are self maps of 1-spheres with type such an $(H)$. Thus, 
application of $\ell_K$ to the equation yields
\[
 \sum_{j\in J}\lambda_j\ell_K(a_j)=0,
\]
where $J\subseteq I$ is the set of all indices such that $a_j$ is represented by a self map of ${\quot GK}_+$. This complex consists of two $0$-cells one of which is the base point, hence 
the cellular chain complex $C_*(\left({\quot GK}_+\right)^K, *)$ is isomorphic to $\setZ\pi_0(W(K))$. The map $r_w$ for $w\in W(K)$ induces the map
\[
 \setZ\pi_0(W(K))\to\setZ\pi_0(W(K)),\;v\mapsto vw.
\]
Therefore, its trace is $[w]\in\setZ Co(\pi_0(W(K)))$. We conclude that
\[
 0=\sum_{j\in J}\lambda_j\ell_K(a_j)=\sum_{j\in J}\lambda_j[w_j],
\]
where $[w_j]$ is the unique class in $Co(\pi_0(W(K)))$ representing $a_j$. By definition of $A'$, the elements $[w_j]$ are linearly independent and so we must have $\lambda_j=0$ for all 
$j\in J$, in contradiction to the choice of $K$. It follows that $\lambda_i=0$ for all $i\in I$. 

The statement concerning the explicit isomorphism follows along the lines of this proof. 
\end{proof}

\begin{corollary}\label{cor:lefschetzaxioms}
 An equivariant Lefschetz number is uniquely characterized by its values on $G$-self maps of finite wedge sums of equivariant 1-spheres. In particular, if $(L^1_G, B)$ is an equivariant 
 Lefschetz number and $L^2_G$ is an assignment of an element $L^2_G(f)\in B$ to every equivariant self map $f:X\to X$ of a finite $G$-CW complex, then $L^1_G(f)=L^2_G(f)$ for every self 
 map $f$ if and only if $L^2_G$ satisfies the following axioms.
 \begin{enumerate}
  \item If $f:X\to X$ is equivariantly homotopic to $h:X\to X$, then $L^2_G(f)=L^2_G(h)$.

  \item If $f:X\to Y$, $h:Y\to X$, then
   \[
    L^2_G(f\circ h)=L^2_G(h\circ f).
   \]

  \item If $A\subseteq X$ is a subcomplex and the diagram
  \[
   \xymatrix{
            A\ar[r]\ar[d]^{f\big|_{A}}&X\ar[r]\ar[d]^f&\quot XA\ar[d]^{\hat{f}}\\
            A\ar[r]&X\ar[r]&\quot XA 
            }
  \]
  commutes, then 
  \[
   L^2_G(f)=L^2_G(f\big|_A)+L^2_G(\hat{f}).
  \]

  \item If $X=\bigvee{\quot G{H_i}}_+\wedge\setS^1$ is a finite wedge sum and $f:X\to X$ is a $G$-map, then
   \[
    L^2_G(f)=L^1_G(f).
   \]
\end{enumerate}
\end{corollary}

\begin{proof}
 The first three axioms imply that $L^2_G$ is an equivariant Lefschetz number. The fact that $L^2_G(f)=L^2_G(h)$ in a diagram
 \[
  \xymatrix{
           X\ar[r]^f\ar[d]^k&X\ar[d]^k\\
           Y\ar[r]^h&Y
           }
 \]
with a $G$-homotopy equivalence $k:X\to Y$ follows from Axioms 1. and 2. since 
\[
 L^2_G(f)=L^2_G(k^{-1}\circ h\circ k)=L^2_G(k^{-1}\circ k\circ h)=L^2_G(h).
\]
We conclude that $L^2_G$ is a homomorphic image of the universal Lefschetz number. Axiom 4. specifies the values of this homomorphism on generators of the universal Lefschetz group and 
consequently, $L^2_G$ and $L^1_G$ are induced by the same homomorphism, therefore equal.
\end{proof}

The following corollary is immediate from the characterization of generators of the universal linear Lefschetz group.

\begin{corollary}\label{cor:linearaxioms}
 If $L^1$ and $L^2$ are both linear, then Axiom 4 can be replaced by
\begin{enumerate}
 \item[4b.] For every closed subgroup $H$ of $G$, $w\in Co(\pi_0(\left(\quot GH\right)^H)$ and $r_w:{\quot GH}_+\to{\quot GH}_+,\;[g]\mapsto[gw]$, 
   \[
    L^2_G(r_w)=L^1_G(r_w).
   \]
 \end{enumerate}
\end{corollary}

Finally, by the identification of a basis of the universal Lefschetz group, we have the following result.

\begin{corollary}
 If $G$ is a topological group such that $W(H)\cong\left(\quot GH\right)^H$ for all closed subgroups $H$ of $G$, then the set of equivariant Lefschetz numbers with values in a group $B$ 
 is in bijective correspondence to the set of set maps 
\[
 \bigcup_{(H)}Co(\pi_0(W(H)))\to B.
\]
 The correspondence is given by $\varphi\mapsto \ell_\varphi$, where $\ell_\varphi$ is the image of the universal linear Lefschetz number under the homomorphism
\[
 U_GL\to B,\;[r_w,{\quot GH}_+]\mapsto\varphi([w]).
\]
\end{corollary}
 
\section{Equivariant Cellular Homology}\label{section:cellhom}
 In this section we briefly sketch the construction of equivariant cellular homology. We mainly follow \cite{ynrohc} and \cite{willson}. Recall that a $G$-CW complex is a $G$-space $X$ 
 that is the colimit of its $n$-skeleta $X_n$. $X_0$ is a disjoint union of $G$-orbits $\quot GH$ and $X_{n+1}$ is constructed from $X_n$ by taking the cofibre of a map from a disjoint
 union of equivariant spheres to $X_n$, i.e. a map
\[
 f:\coprod_{i\in I}\quot G{H_i}\times\setS^n\to X_n.
\]
 The restriction of $f$ to a summand is called an attaching map. $X$ is said to be finite, if $X=X_n$ for some $n$ and only finitely many spheres are attached at each stage. A based $G$-CW
 complex is a $G$-CW complex with a $G$-fixed base point. We do not assume attaching maps to be based. 
  
 We are going to describe a chain complex associated with the filtration of $X$ by its skeleta that is analogous to the cellular chain complex of an ordinary CW complex. For missing
 details, we refer to \cite{willson}. The basic idea is to replace the summand $\setZ$ in the cellular complex corresponding to a cell by the functor $\setZ[\bullet\,,\quot GH]_G$, 
 corresponding to an equivariant cell of type $H$. In the following, $\calO$ will denote a full subcategory of the category of $G$-orbits and $G$-homotopy classes, consisting of finitely 
 many objects. $\calA$ will denote the category of abelian groups.
 
 \begin{definition}
  Let $X$ be a $G$-CW complex. We define the $n$-th cellular chain functor to be the functor
\[
 C^G_n(X):\calO\to\calA,\;C^G_n(X)(\quot GK)=\bigoplus_H\bigoplus_{\stiny{\begin{array}{c}n-\mbox{cells}\\\mbox{of type }H\end{array}}}\setZ[\quot GK, \quot GH]_G.
\]
 For a $G$-map $\varphi:\quot GK\to\quot GL$, $C^G_n(X)(\varphi)$ is defined as post-composition with $\varphi$ on generators.
 \end{definition}

 The functors $C_*^G(X)$ constitute a chain complex of contravariant functors $\calO\to\calA$. The natural transformation $d:C^G_n(X)\to C^G_{n-1}(X)$ is defined as follows. Let $SX$ be 
 the unreduced suspension of a $G$-space. That is, $SX=\quot{X\times\incc{0,1}}\sim$, where $X\times\{0\}$ is identified to a point and $X\times\{1\}$ is identified to another point. We 
 regard $SX$ as a based $G$-space with base point $X\times\{1\}$. Let $\sigma$ be an $n$-cell of $X$, $\tau\cong{\quot GK}_+\wedge\setS^{n-1}$ an $(n-1)$-cell. Let 
 $\alpha:\quot GH\times\setS^{n-1}\to X_{n-1}$ be the attaching map of $\sigma$. $\alpha$ induces a map
 \[
  \quot GH\times S\setS^{n-1}\to SX_{n-1},\;([g], [x, t])\mapsto[\alpha([g], x), t].
 \]
 Since all points $([g], [x, t])$ with $t=1$ are mapped to the base point of $SX_{n-1}$, this in turn induces a map
 \[
  {\quot GH}_+\wedge\setS^n\to SX_{n-1},
 \]
 which we can follow with the quotient map to 
\[
 \quot{SX_{n-1}}{SX_{n-2}}\cong\Sigma\left(\quot{X_{n-1}}{X_{n-2}}\right).
\]
 Together with the projection $p_\tau$ to $\tau$, we obtain a $G$-map
\[
 k:{\quot GH}_+\wedge\setS^n\to\Sigma\left(\quot{X_{n-1}}{X_{n-2}}\right)\slra{\Sigma p_\tau}{\quot GK}_+\wedge\setS^n.
\]
 $k$ induces a map $\setS^n\to\setS^n$ in the group quotient. Let $-1\in\setS^{n-1}\subseteq S\setS^n$, embedded as $\setS^{n-1}\times\{\frac12\}$. We find an equivariant homotopy of $k$ 
 such that the induced map is transverse to $-1\in\setS^n$. This implies that the preimage of $-1$ consists of finitely many points $x_1,\dots, x_{m_{\sigma,\,\tau}}\in\setS^n$. Assuming 
 that $k$ itself has this property, we obtain that the preimage of $\quot GK\times\{-1\}$ under $k$ is $\quot GH\times\{x_1,\dots, x_{m_{\sigma,\,\tau}}\}$. We define
 \[
  \psi^{\sigma,\,\tau}_j:\quot GH\to\quot GK,\;\psi_j([g])=k([g], x_j),
 \]
 where we identify the point $k([g], x_j)=([\tilde{g}], -1)$ with its first component. Furthermore, let $d^{\sigma,\,\tau}_j$ be the local degree of the map induced by $k$ at $x_j$. Let 
 $\varphi:\quot GL\to\quot GH$ be any $G$-map, then
 \[
  d([\varphi]\cdot\sigma)=\sum_\tau\sum_{j=1}^{m_{\sigma,\,\tau}} d^{\sigma,\,\tau}_j[\psi^{\sigma,\,\tau}_j\circ\varphi]\cdot\tau,
 \]
 and we extend linearly to all of $C^G_n(X)(\quot GH)$. It is clear that $d$ constitutes a natural transformation. The fact that $d^2=0$ follows from the fact that forgetting the orbit 
 maps yields the cellular chain functor of the quotient space $\quot XG$.

 If $M:\calO\to\calA$ is any covariant functor, we obtain the equivariant chain complex with coefficients in $M$ as
 \[
  C^G_*(X; M)=C^G_*(X)\tensor_\calO M.
 \]
 The differential is given by $d\tensor\id$. The homology of this chain complex is denoted by $H_n^G(X;M)$ and is called the $n$-th equivariant cellular homology of $X$. The exact 
 definition of the tensor product $\tensor_\calO$ is not important here, because we will find a more convenient description of $C_*^G(X, M)$ later. 

 To turn cellular homology into a functor, we have to specify $C^G_n$ as a functor from the category of $G$-CW complexes to the category of contravariant functors $\calO\to\calA$.
 Thus, we have to define a natural transformation $C^G_n(f):C^G_n(X)\to C^G_n(Y)$ for a $G$-map $f:X\to Y$. It is defined similarly to the boundary operator. Let 
 $\iota_i:{\quot G{H_i}}_+\wedge\setS^n\to\quot{X_n}{X_{n-1}}$ be the inclusion of the $i$-th cell of $X_n$ and let $p_j:\quot{Y_n}{Y_{n-1}}\to{\quot G{K_j}}_+\wedge\setS^n$ be the 
 projection onto the $j$-th cell of $Y$. We can assume that $f$ is cellular. Then we obain a map
 \[
  k:{\quot G{H_i}}_+\wedge\setS^n\slra{\iota_i}\quot{X_n}{X_{n-1}}\slra{f}\quot{Y_n}{Y_{n-1}}\slra{p_j}{\quot G{K_j}}_+\wedge\setS^n.
 \]
 As above, we can assume that $k$ induces a map $\tilde{k}$ transverse to $-1$ in the group quotient and thus, $k$ induces finitely many $G$-maps 
 \[
  \psi^{\sigma,\,\tau}_\ell:\quot G{H_i}\to\quot G{K_j},\;\ell=1,\dots, m.  
 \]
 Let $d_\ell^{\sigma,\,\tau}$ be the local degree of the map $\tilde{k}$ at $x_j$ and $\varphi:\quot GL\to\quot G{H_i}$ any $G$-map. Then we define
 \[
  C^G_n(f)([\varphi]\cdot\sigma)=\sum_\tau\sum_{\ell=1}^{m_{\sigma,\,\tau}} d^{\sigma,\,\tau}_\ell[\psi^{\sigma,\,\tau}_\ell\circ\varphi]\cdot\tau
 \]
 and extend linearly. Note again that by forgetting the orbit map part, we end up with the map $f$ induces in the cellular homology of the quotient.

 For a coefficient system $M$, the map $C_*^G(f; M)$ is defined as $C_*^G(f)\tensor\id_M$.

 We will now specify the coefficients we will use for cellular homology. Let $\calI$ be the free abelian group generated by all homotopy classes of $G$-maps between orbits. There is
 an obvious ring structure on $\calI$, given by composition of homotopy classes whenever composition is possible, and being $0$ otherwise. $\calI$ is called the isotropy ring associated
 to $G$ (or $\calO$). We can turn $\calI$ into a coefficient system as follows.

 The value of $\calI$ on $\quot GH$ is the ideal $\id_{\quot GH}\cdot\calI$. For a $G$-map $\varphi:\quot GH\to\quot GK$, $\calI(\varphi)$ is pre-composition with $\varphi$.
 Since $\id_{\quot GH}$ is an idempotent element in $\calI$, $\calI(\quot GH)$ is a projective right $\calI$-module. Because of the isomorphism
\[
 C_*^G(X)\tensor_\calO\calI\cong\bigoplus_H\calI(\quot GH)
\]
 stemming from the dual Yoneda isomorphism, the equivariant cellular chain complexes of $X$ inherit the structure of projective right $\calI$-modules as well. 

\section{Traces}\label{section:traces}
In order to be able to define the homological equivariant Lefschetz number, we need some background in the theory of traces for module endomorphisms of finitely generated projective
modules over arbitrary rings. The general theory was developed by Hattori in \cite{hattori}, from which much of the following is taken. The resulting trace is called the Hattori-Stallings 
trace. We are especially interested in the behaviour of this trace with respect to homology.

In the following, $R$ will denote an associative ring and by $R$-module we will mean a right $R$-module. For an $R$-module $M$, we let $DM=\hom_R(M, R)$ be the dual module. Furthermore, 
we write $[M, N]$ for $\hom_R(M, N)$, the $R$-module homomorphisms from $M$ to $N$.
\begin{definition}
 Let $P$ be a finitely generated projective $R$-module. Let $R^{ab}$ be the abelianization of $R$ and let $\pi:R\to R^{ab}$ be the projection. The canonical map
\[
 \vartheta:P\tensor DP\to[P,P],\;\vartheta(x\tensor\xi)(y)=x\cdot\xi(y)
\]
is an isomorphism and we can define
\[
 \trace:[P,P]\to R^{ab},\;f\mapsto\pi\circ\eps\circ\vartheta^{-1}(f),
\]
where $\eps:P\tensor DP\to R$ is the evaluation map. $\trace(f)$ is called the Hattori-Stallings trace of $f$.
\end{definition}

\begin{proposition}\label{prop:tracecommutative}
 The Hattori-Stallings trace is commutative in the following sense. If $f:P\to Q$, $g:Q\to P$ are module homomorphisms, then 
\[
 \trace(f\circ g)=\trace(g\circ f).
\]
\end{proposition}

\begin{proof}
 We define maps
\[
 \vartheta_{Q, P}:Q\tensor DP\to[P, Q],\;\vartheta_{Q,P}(y\tensor\xi)(p)=y\cdot\xi(p).
\]
Note that $\vartheta=\vartheta_{P,P}$. Furthermore, we define
\[
 \beta_{Q, P}:(Q\tensor DP)\times(P\tensor DQ)\to P\tensor DP,\;(y\tensor\xi, x\tensor\eta)\mapsto y\cdot\xi(x)\tensor\eta.
\]
Then the following diagram is easily seen to be commutative, where $\tau$ denotes the twisting map.
\[
 \sfoot{\xymatrix{[Q,Q] & [P, Q]\tensor[Q, P]\ar[r]^\tau\ar[l]_\circ & [Q,P]\tensor[P,Q]\ar[r]^\circ & [P,P]\\
           Q\tensor DQ\ar[u]_\vartheta\ar[rd]_{\pi\circ\eps} & (Q\tensor DP)\tensor(P\tensor DQ)\ar[l]_{\beta_{P,Q}\quad\quad}\ar[u]_{\vartheta_{Q, P}\tensor\vartheta_{P, Q}}\ar[r]^\tau&
           (P\tensor DQ)\tensor(Q\tensor DP)\ar[u]^{\vartheta_{P,Q}\tensor\vartheta_{Q, P}}\ar[r]^{\quad\quad\beta_{Q, P}}&Q\tensor DQ\ar[u]^\vartheta\ar[ld]^{\pi\circ\eps}\\
           &R^{ab}\ar[r]^\id&R^{ab}&}}
\]
The claim follows immediately.
\end{proof}

\begin{lemma}\label{lem:ffromid}
 Let $P$ be a finitely generated projective module and $f:P\to P$ an endomorphism. If 
\[
 \vartheta^{-1}(\id_P)=\sum x_i\tensor\xi_i\in P\tensor DP,
\]
then
\[
 \vartheta^{-1}(f)=\sum f(x_i)\tensor\xi_i.
\]
\end{lemma}

\begin{proof}
 We calculate
\begin{eqnarray*}
 \vartheta\left(\sum f(x_i)\tensor\xi_i\right)(y)&=&\sum f(x_i)\cdot\xi_i(y)\\
                                      &=&f\left(\sum x_i\cdot\xi_i(y)\right)\\
                                      &=&f(y),
\end{eqnarray*}
and the claim follows.
\end{proof}

\begin{lemma}\label{lem:traceadditivity}
 Let $0\to A\to B\to C\to 0$ be an exact sequence of finitely generated projective $R$-modules and $f$ a morphism of this exact sequence, then $\trace{f_B}=\trace{f_A}+\trace{f_C}$.
\end{lemma}

\begin{proof}
 Since all involved modules are projective, the exact sequence splits and we can view $B$ as the direct sum $B=A\oplus C$. Let $i_A:A\to B$, $i_C:C\to B$ be the inclusions. Furthermore, we
 have the projections $p_A:B\to B$, $p_C:B\to B$ onto $A$ and $C$, respectively, and these are idempotent elements in $[B, B]$, satisfying $p_A+p_C=\id_B$. Consider the following four 
 morphisms.
 \[
  p_A\tensor Di_A:B\tensor DB\to A\tensor DA,\;x\tensor\xi\mapsto p_A(x)\tensor\xi\circ i_A,
 \]
\[
 p_C\tensor Di_C:B\tensor DB\to C\tensor DC,\;y\tensor\eta\mapsto p_C(y)\tensor\eta\circ i_B,
\]
\[
 [B, B]\to[A,A],\;\varphi\mapsto p_A\circ\varphi\circ i_A,
\]
\[
 [B,B]\to[C,C],\;\varphi\mapsto p_C\circ\varphi\circ i_C.
\]
These constitute a commutative diagram
\[
 \xymatrix{B\tensor DB\ar[rr]^{p_A\tensor Di_A}\ar[d]_{\vartheta_{B}}&&A\tensor DA\ar[d]^{\vartheta_{A}}\\
                [B,B]\ar[rr]&&[A,A]}
\]
and similarly with $A$ exchanged for $C$. Since $\id_B$ maps to $\id_A$ under the lower map, if we write
\[
 \vartheta_{B}^{-1}(\id_{B})=\sum x_i\tensor\xi_i,
\]
we obtain
\[
 \vartheta_A^{-1}(\id_A)=\sum p_A(x_i)\tensor\xi_i\circ i_A,\;\vartheta_C^{-1}(\id_C)=\sum p_C(x_i)\tensor\xi_i\circ i_C.
\]
 We have $i_A\circ f(a)=f(a)$ and $i_C\circ f(c)=f(c)$ for $a\in A$, $c\in C$. Thus we compute
\begin{eqnarray*}
 \trace{f_A}+\trace{f_B}&=&\pi\circ\eps_A\circ\vartheta_A^{-1}(f_A)+\pi\circ\eps_C\circ\vartheta_C^{-1}(f_C)\\
                        &=&\pi\left(\sum\xi_i\circ i_A\circ f\circ p_A(x_i)\right)+\pi\left(\sum\xi_i\circ i_C\circ f\circ p_C(x_i)\right)\\
                        &=&\pi\left(\sum\xi_i\circ f\circ p_A(x_i)\right)+\pi\left(\sum\xi_i\circ f\circ p_C(x_i)\right)\\
                        &=&\pi\left(\sum\xi_i\circ f(p_A(x_i)+p_C(x_i))\right)\\
                        &=&\pi\left(\sum\xi_i\circ f(x_i)\right)\\
                        &=&\pi\circ\eps_B\circ\vartheta_B^{-1}(f)\\
                        &=&\trace{f}.
\end{eqnarray*}
\end{proof}

\section{The Homological Lefschetz Number}
Let $\setU_G$ be the tom Dieck ring of $G$. That is,
\[
 \setU_G=\bigoplus_{(H)}\setZ\cdot(H),
\]
where the sum ranges over all conjugacy classes of closed subgroups of $G$. $\setU_G$ is the domain of the universal equivariant Euler characteristic, constructed similarly to the
universal equivariant Lefschetz group, compare \cite{tomdieck}. This will however not be of concern in the following. 

Recall the definition of the isotropy ring $\calI$ from Section \ref{section:cellhom}. There is an obvious augmentation map $\eps:\calI\to\setU_G$, defined as follows. Let 
$\varphi:\quot GH\to\quot GK$ be a $G$-map. Then the basis element $[\varphi]$ of $\calI$ maps to $1\cdot(H)$ if and only if $H=K$, and to $0$ otherwise. With the help of this 
augmentation, we can define the topological equivariant Lefschetz number.

\begin{definition}
 Let $X$ be finite a $G$-CW complex, $f:X\to X$ a $G$-map. The unreduced homological equivariant Lefschetz number of $f$ is defined to be
\[
 \Lambda_G(f)=\sum_{n=0}^\infty(-1)^n\cdot\eps(\trace C^G_n(f)).
\]
The (reduced) homological equivariant Lefschetz number is defined as 
\[
 L_G(f)=\Lambda_G(f)-1\cdot(G).
\]
\end{definition}

It is not hard to see that the homological equivariant Lefschetz number is indeed a linear equivariant Lefschetz number in the sense of the definition.

\begin{proposition}\label{prop:homlefaxioms}
 The homological equivariant Lefschetz number satisfies the first three axioms of Corollary \ref{cor:lefschetzaxioms} and is linear.
\end{proposition}

\begin{proof}
 The homotopy axiom is obvious. The commutativity axiom follows from the commutativity of the trace, Proposition \ref{prop:tracecommutative}. Let $X, Y$ be finite $G$-CW complexes and let
 $f:X\to Y$, $h:Y\to X$ be $G$-maps. Then
 \begin{eqnarray*}
  \Lambda_G(f\circ h)&=&\sum_{n=0}^\infty(-1)^n\trace(C_n^G(f\circ h;\calI))\\
                     &=&\sum_{n=0}^\infty(-1)^n\trace(C_n^G(f;\calI)\circ C_n^G(h;\calI)))\\
                     &=&\sum_{n=0}^\infty(-1)^n\trace(C_n^G(f;\calI))\cdot\trace(C_n^G(h;\calI))\\
                     &=&\sum_{n=0}^\infty(-1)^n\trace(C_n^G(h;\calI))\cdot\trace(C_n^G(f;\calI))\\
                     &=&\Lambda_G(h\circ f).
 \end{eqnarray*}
 This clearly carries over to the reduced case. 

 For the cofibration axiom, consider a cofibre sequence $A\slra{i} X\slra{p}\quot XA$. We have the short exact sequences
 \[
  0\to C^G_n(A;\calI)\to C^G_n(X;\calI)\to\quot{C_n^G(X;\calI)}{C_n^G(A; \calI)}\to0,
 \]
 and for $n\geq1$, $\quot{C_n^G(X;\calI)}{C_n^G(A; \calI)}$ is isomorphic to $C_n^G(\quot XA;\calI)$. For $n=0$, 
 \[
 \quot{C_0^G(X;\calI)}{C_0^G(A; \calI)}=\bigoplus_\sigma\bigoplus_K\setZ[\quot GK,\quot GH]_G,
 \]
 where the first sum ranges over the $0$-cells of type $H$ which are not contained in $A_0$. On the other hand,
 \[
  C_0^G(\quot XA;\calI)=\quot{C_0^G(X;\calI)}{C_0^G(A; \calI)}\oplus\bigoplus_K\setZ[\quot GK,\quot GG]_G,
 \]
 the last summand representing the base point. The above short exact sequence for $n=0$ yields
 \[
  \trace(C^G_0(f))=\trace(C^G_0(f\big|_A))+\trace(\tilde{f}),
 \]
 where $\tilde{f}$ is the map $C^G_0(f)$ induces in the quotient. From the short exact sequence
 \[
  0\to C_0^G(*;\calI)\to C_0^G(\quot XA;\calI)\to\quot{C_0^G(X;\calI)}{C_0^G(A; \calI)}\to0,
 \]
 we have $\trace(C^G_0(\hat{f}))=\trace(C^G_0(*))+\trace(\tilde{f})=1\cdot(G)+\trace(\tilde{f})$. Putting all this together yields
 \begin{eqnarray*}
  \Lambda_G(f)&=&\sum_{n=0}^\infty(-1)^n\trace(C_n^G(f;\calI))\\
        &=&\sum_{n=1}^\infty(-1)^n\trace(C^G_n(f\big|_A;\calI))+\sum_{n=1}^\infty(-1)^n\trace(C^G_n(\hat{f};\calI))\\
        &&+\trace(C_0^G(f;\calI))+\trace(C_0^G(\hat{f};\calI))-1\cdot(G)\\
        &=&\Lambda_G(f\big|_A)+\Lambda_G(\hat{f})-1\cdot(G).
 \end{eqnarray*}
 This proves the cofibration axiom for the reduced variant.

 Finally for linearity, we prove that additivity holds on the wedge sum of two complexes. By Propositions \ref{prop:wedgeadditivity} and \ref{prop:axiomslinwedge}, this is more than 
 sufficient to prove the claim. We have 
 \[
  C_n^G(X\vee Y;\calI)=C_n^G(X;\calI)\oplus C_n^G(Y;\calI)
 \]
 for $n>0$. For $n=0$, $C_0^G(X\vee Y;\calI)$ is the free abelian group with basis the sets $[\quot GK, \quot GH]_G$ for each $0$-cell of $X$ or $Y$ of type $H$ which is not the basepoint,
 together with basis elements $[\quot GK, \quot GG]_G$ corresponding to the common base point of $X$ and $Y$. The canonical map $C_0(X;\calI)\oplus C_0(Y;\calI)\to C_0(X\vee Y;\calI)$ has 
 kernel $C_0(*;\calI)$ and we obtain the exact sequence
 \[
  0\to C_0^G(*;\calI)\to C_0^G(X;\calI)\oplus C_0^G(Y;\calI)\to C_0^G(X\vee Y;\calI)\to0.
 \]
 Consequently, $\trace{C^G_0(f)}+\trace{C^G_0(*)}=\trace{C_0^G(f_X)}+\trace{C_0^G(f_Y)}$, where $*:*\to *$ is the unique $G$-map. Clearly, this map has trace $1\cdot(G)$, and hence,
 \begin{eqnarray*}
  L_G(f)&=&\Lambda_G(f)-1\cdot(G)\\
        &=&\sum_{n=0}^\infty(-1)^n\eps(\trace{C_n^G(f)})-1\cdot(G)\\
        &=&\sum_{n=1}^\infty(-1)^n\left(\eps(\trace{C_n^G(f_X)})+\eps(\trace{C_n^G(f)_Y})\right)\\
        & &+\eps(\trace{C_0^G(f_X)})+\eps(\trace{C_0^G(f_Y)})-2\cdot(G)\\
        &=&L_G(f_X)+L_G(f_Y).
 \end{eqnarray*}
\end{proof}

We proceed to calculate the value of $L_G$ on finite wedge sums of equivariant 1-spheres in order to be able to compare it, by means of Corollary \ref{cor:lefschetzaxioms}, to other 
equivariant Lefschetz numbers.

\begin{proposition}\label{prop:homlefonspheres}
 Let $X=\bigvee_{i\in I}{\quot G{H_i}}_+\wedge\setS^1$ be a finite wedge of equivariant 1-spheres and $f:X\to X$ a $G$-map. Let $f_i$ be the map $f$ induces on the $i$-th summand via 
 inclusion and projection. Then
\[
 L_G(f)=\sum_{i\in I}L(\conj{f}_i)\cdot(H),
\]
 where $\conj{f}_i$ is the map $f_i$ induces in the group quotient and $L$ denotes the ordinary reduced Lefschetz number.
\end{proposition}

\begin{proof}
By linearity, Proposition \ref{prop:homlefaxioms}, it suffices to assume that $X$ is a single equivariant 1-sphere. We can assume for the rest of the proof that $\conj{f}$ is transverse 
regular at $-1\in\setS^1$. The map induced by $\conj{f}$ in first homology is multiplication by the degree of $\conj{f}$, which is the sum of local degrees, computed at the preimages of 
$-1$. To write down a formula, let $\conj{f}^{-1}(-1)=\{x_1,\dots, x_m\}$ and $d_j$ be the local degree of $\conj{f}$ at $x_j$, $j=1, \dots, m$, then
\[
 H_1(\conj{f})(\alpha)=\sum_{j=1}^md_j\cdot\alpha. 
\]
It follows that the Lefschetz number of $\conj{f}$ is $1-\sum_{j=1}^md_j$ and so for the reduced Lefschetz number we have $L(\conj{f})=-\sum_{j=1}^md_j$.

For the equivariant Lefschetz number, we note that 
\[
C_1^G(X;\calI)=\bigoplus_{K}\setZ[\quot GK, \quot GH]_G. 
\]
Let $\psi_j:\quot GH\to\quot GH$ be the $G$-map obtained by restricting $f$ to $\quot GH\times\{x_j\}$, $j=1,\dots, m$. By definition, the induced map in the first cellular chain module 
of $X$ is given as
\[
 C^G_1(f;\calI)([\varphi])=\sum_{j=1}^md_j\cdot[\psi_j\circ\varphi],
\]
where $\varphi:\quot GK\to\quot GH$ is any $G$-map. Now let
\[
 x=[\id_{\quot GH}]\in C_1^G(X; \calI)
\]
and
\[
 \xi:C_1^G(X; \calI)\to\calI,\;[\varphi]\mapsto[\varphi],
\]
extended linearly. Since $[\varphi]=\xi(x\cdot[\varphi])$, we have that
\[
 \vartheta^{-1}(\id_{C_1^G(X;\calI)})=x\tensor\xi.
\]
Thus, $\vartheta^{-1}(C_1^G(f;\calI))=C_1^G(f;\calI)(x)\tensor\xi$ by Lemma \ref{lem:ffromid}, and the trace of $C_1^G(f;\calI)$ is calculated as
\[
 \trace(C_1^G(f;\calI))=\xi(C_1^G(f;\calI)(x))=\sum_{j=1}^md_j[\psi_j]\in\calI^{ab}.
\]
The augmentation sends all elements $[\psi_j]$ to $1\cdot(H)$, since these are self-maps of $\quot GH$. Therefore,
\[
 \eps(\trace(C_1^G(f;\calI)))=\sum_{j=1}^md_j\cdot(H).
\]
Clearly, $\eps(\trace(C_0^G(f;\calI)))=1\cdot(G)$, since we have a single $0$-cell of type $(G)$. We conclude that
\[
 \Lambda_G(f)=1\cdot(G)-\sum_{j=1}^md_j\cdot(H),
\]
which proves that indeed $L_G(f)=L(\conj{f})\cdot(H)$.
\end{proof}

\begin{corollary}
 The homological equivariant Lefschetz number is the unique linear equivariant Lefschetz number $L_G$ such that $L_G(\varphi)=1\cdot(H)$ for every self map $\varphi$ of ${\quot GH}_+$.
\end{corollary}
\begin{proof}
 The quotient of the suspension of a $G$-map of orbits is the identity on $\setS^1$, so its reduced non-equivariant Lefschetz number is $-1$. The claim follows from Proposition
 \ref{prop:homlefonspheres} and Corollary \ref{cor:linearaxioms}. 
\end{proof}

\section{The Analytical Lefschetz Number}
In this section we will define a homotopy invariant on $G$-ANRs that will be seen to satisfy the axioms of Corollary \ref{cor:lefschetzaxioms}. It was originally constructed by Dzedzej in
\cite{dzedzej}. We briefly recall some facts from the theory of ANRs and $G$-ANRs.

\begin{definition}
 An absolute neighbourhood retract, or ANR, is a metric space $X$ satisfying the following universal property. Whenever $i:X\to M$ is an isometric embedding of $X$ into a metric space $M$,
 there is a neighbourhood $U\subseteq M$ of $i(X)$ and a retraction $r:U\to X$.
\end{definition}

The notion of a $G$-ANR is apparent. The basic theory of $G$-ANRs was developed by Murayama in \cite{murayama}. There he proves the following proposition. Recall the notion of isotropy 
subgroup of an element $x$ in a $G$-space $X$. It is the group $G_x=\{g\in G\;|\;g.x=x\}$. We denote
\begin{eqnarray*}
 X_{(H)}&=&\{x\in X\;|\;(G_x)=(H)\}\\
 X_{\geq(H)}&=&\{x\in X\;|\;(G_x)\geq(H)\}\\
 X_{>(H)}&=&\{x\in X_{\geq(H)}\;|\;(G_x)\neq(H)\},
\end{eqnarray*}
the partial order, as before, being subconjugacy of conjugacy classes of closed subgroups.

\begin{proposition} Let $H$ be a closed subgroup of $G$ and let $X$ be a $G$-ANR. Then the spaces $X_{(H)}, X_{\geq(H)}$ and $X_{>(H)}$ are $G$-ANRs and the quotient $\quot XG$ is an 
ordinary ANR. Furthermore, any finite $G$-CW complex is a $G$-ANR.
\end{proposition}
   
Let $X\subseteq U$ be a pair of ANRs, $f:X\to U$ a continuous map having no fixed points on $\partial X$. For simplicity we assume that $U$ can be embedded into $\setR^n$ via an embedding 
$i:U\to\setR^n$, otherwise we have to put some compactness assumption on $f$. Now $i$ embeds $X$ as well and so we find a neighbourhood $V$ of $i(X)$ and a retraction $r:V\to i(X)$. The 
map $i\circ f\circ i^{-1}\circ r:V\to V$ is well-defined and has no fixed points outside of $i(X)$. Hence, its fixed point index (see \cite{nussbaum}) is well-defined and we write 
$i(f, X, U)$ for this number. It is shown in \cite{nussbaum} that this index is well-defined and enjoys all the usual properties of an index.
      
Next we take $X\subseteq U$ to be a pair of $G$-ANRs, $f:X\to U$ a $G$-map without fixed orbits on $\partial X$. Since $\quot XG$ and $\quot UG$ are ANRs, the fixed point index of the 
induced map $\conj{f}:\quot XG\to\quot UG$ is defined and so are the various indices of the induced maps $\conj{f_{\geq(H)}}:\quot{X_{\geq(H)}}G\to\quot{U_{\geq(H)}}G$ and 
$\conj{f_{>(H)}}:\quot{X_{>(H)}}G\to\quot{U_{>(H)}}G$ for closed subgroups $H\subseteq G$. So to each orbit type $(H)$ of $X$, we can assign the integer
\[
 i_{(H)}(f, X, U)=i(\conj{f_{\geq(H)}}, \quot{X_{\geq(H)}}G, \quot{U_{\geq(H)}}G)-i(\conj{f_{>(H)}}, \quot{X_{>(H)}}G, \quot{U_{>(H)}}G). 
\]
       
Assume $X\subseteq U$ are $G$-ANRs with finite orbit type. Then the fixed orbit index of $f$ is defined to be the element
\[
 i_G(f, X, U)=\sum_{(H)}i_{(H)}(f)\cdot(H)\in\setU_G,
\]
In particular, if $X=U$ is a finite $G$-CW complex, the fixed orbit index of a $G$-map $f:X\to X$ is defined. In that case we write $i_G(f, X)$ for $i_G(f, X, X)$.
      
We have augmentation maps 
\[
\eps_{(H)}=U_G\to\setZ,\;\eps_{(H)}(a)=\sum_{(K)\geq(H)}\pi_{(K)}(a), 
\]
where $\pi_{(K)}:U_G\to\setZ$ is the projection onto the $(K)$-th summand. Thus, $\eps{(H)}$ takes all the coefficients of orbit types less or equal to $(H)$ and sums them up. 

\begin{definition}
 Let $X$ be a $G$-CW complex and $f:X\to X$ a $G$-map. The reduced fixed orbit index of $f$ is defined to be the element
\[
 L_G(f)=i_G(f, X)-1\cdot(G).
\]
\end{definition}

We note that in fact the reduced fixed orbit index can be defined as
\[
 L_G(f)=\sum_{(H)}\left(\hat{i}(\conj{f_{\geq(H)}}, \quot{X_{\geq(H)}}G)-\hat{i}(\conj{f_{>(H)}}, \quot{X_{>(H)}}G)\right)\cdot(H),
\]
where $\hat{i}$ is the reduced fixed orbit index, i.e. $\hat{i}(f)=i(f)-1$. Obviously, all the terms coming from the $-1$-part cancel out except for the last one corresponding to $(G)$, 
which shows that the two definitions are equal. Since the fixed point index equals the Lefschetz number on finite CW complexes, the last equality takes the more convenient form
\[
 L_G(f)=\sum_{(H)}\left(L(\conj{f}_{\geq(H)})-L(\conj{f}_{>(H)})\right)\cdot(H)
\]

\begin{proposition}
 The reduced fixed orbit index satisfies the first three axioms of Corollary \ref{cor:lefschetzaxioms}.
\end{proposition}

\begin{proof}
This follows immediately from the fact that the non-equivariant Lefschetz number satisfies the corresponding axioms with trivial group action.
\end{proof}

We compute the value of the fixed orbit index on a wedge sum of equivariant $1$-spheres. Along the lines it turns out that the fixed orbit index is a linear equivariant Lefschetz number,
which of course could also have been proven directly.

\begin{proposition}
 Let $X=\bigvee_{i\in I}{\quot G{H_i}}_+\wedge\setS^1$ be a finite wedge sum of equivariant 1-spheres and $f:X\to X$ a $G$-map. Then
 \[
  L_G(f)=\sum_{i\in I}L(\conj{f}_i)\cdot(H_i),
 \]
 where $f_i$ is the self map of ${\quot G{H_i}}_+\wedge\setS^1$ induced by $f$ via inclusion of and projection to the $i$-th summand and $\conj{f}_i$ is the map $f_i$ induces in the group 
 quotient. In particular, the fixed orbit index equals the homological equivariant Lefschetz number by Corollary \ref{cor:lefschetzaxioms}.
\end{proposition}

\begin{proof}
 For any closed subgroup $H$ of $G$, we have that $X_{>(H)}$ is an invariant subcomplex of $X_{\geq(H)}$ and their quotient is the wedge sum of those equivariant spheres 
 ${\quot G{H_i}}_+\wedge\setS^1$ such that $(H_i)=(H)$. Passing to the group quotient and using the cofibration axiom for the ordinary reduced Lefschetz number, this yields
\[
 L(\conj{f}_{\geq(H)})=L(\conj{f}_{>(H)})+L(\conj{f}_{(H)}).
\]
 Inserting into the definition of the reduced fixed orbit index, we obtain
\begin{eqnarray*}
 L_G(f)&=&\sum_{(H)}\left(L(\conj{f}_{\geq(H)})-L(\conj{f}_{>(H)})\right)\cdot(H)\\
       &=&\sum_{(H)}L(\conj{f}_{(H)})\cdot(H).
\end{eqnarray*}
Since the non-equivariant reduced Lefschetz number is additive on wedge sums, we have
\[
 L(\conj{f}_{(H)})=\sum_i L(\conj{f}_i),
\]
where the sum runs over all indices $i$ such that $(H_i)=(H)$ and $f_i$ is defined as above. This proves the claim.
\end{proof}

\section{The Lefschetz Fixed Orbit Theorem}
The theorem is easily deduced from the analytical definition of the Lefschetz number. A topological proof is also possible.

\begin{theorem}
 Let $X$ be a finite $G$-CW complex and $f:X\to X$ a $G$-map. Let $\pi_{(H)}:U_G\to\setZ$ be the projection to the $(H)$-component of $U_G$. If $(H)\neq G$ and $\pi_{(H)}(L_G(f))\neq0$,
 $f$ has a non-trivial fixed orbit of orbit type at least $(H)$. If the base point $*\in X$ is an isolated point, then the results holds for $(H)=(G)$ as well.
\end{theorem}

\begin{proof}
 By definition of the analytical equivariant Lefschetz number,
 \[
  \pi_{(H)}(L_G(f))=i(\conj{f_{\geq(H)}}, \quot{X_{\geq(H)}}G)-i(\conj{f_{>(H)}}, \quot{X_{>(H)}}G)
 \]
 for $(H)\neq(G)$ and
\[
 \pi_{(G)}(L_G(f))=i(\conj{f_{\geq(G)}}, \quot{X_{(G)}}G)-1.
\]
If the former is non-zero, then either 
\[
i(\conj{f_{\geq(H)}}, \quot{X_{\geq(H)}}G)
\]
or 
\[
 i(\conj{f_{>(H)}}, \quot{X_{>(H)}}G)
\]
will be non-zero. Both facts imply that $\conj{f}$ has a fixed point in the set of points with orbit type at least $(H)$, and this implies existence of a fixed orbit of $f$ of type at 
least $(H)$. Non-triviality follows since if the base point were the only fixed orbit, we would have 
\[
i(\conj{f_{\geq(H)}}, \quot{X_{\geq(H)}}G)=i(\conj{f_{>(H)}}, \quot{X_{>(H)}}G).
\]
Under the additional assumption that the base point is isolated, non-vanishing of $i(\conj{f_{\geq(G)}}, \quot{X_{(G)}}G)-1$ implies that there must be at least one more fixed point of 
$f$ in $X_{(G)}$, apart from the base point, since $i(\conj{f_{\geq(G)}}, \quot{X_{(G)}}G)$ is different from $1$ and the local index of the base point is $1$.
\end{proof}

We give a simple example that shows that the term 'at least' can not be deleted in the statement of the theorem. Consider the $\setZ_2$-complex $\setS^1\subseteq\setC$, with the action 
given by reflection at the imaginary axis. The quotient space can be identified with $\incc{-1,1}$ in the obvious way. Any continuous map $\incc{-1,1}\to\incc{-1,1}$ fixing $-1$ and $1$ 
lifts to an equivariant self map of $\setS^1$. For example, we can consider the map
\[
 f_\mu:\incc{-1,1}\to\incc{-1,1},\;f_\mu(t)=t^3-\mu t^2+\mu
\]
for $\mu\in\incc{0,1}$. $f$ induces (not uniquely) an equivariant homotopy $F$ of $G$-self maps of $\setS^1$. By definition of the fixed orbit index, we have to calculate the Lefschetz 
number of $f_\mu$ in order to obtain the equivariant Lefschetz number of $F_\mu$. A simple calculation yields that $f_0$ has Lefschetz number $1$. The restriction to $\{-1,1\}$, which is
the quotient of the $\setZ_2$-fixed space os $\setS^1$, is the identity on a finite set, hence its Lefschetz number is $2$. We obtain that
\[
 i_G(F_\mu, \setS^1)=-(e)+2\cdot(\setZ_2),
\]
and therefore,
\[
 L_G(F_\mu)=-(e)+(\setZ_2).
\]
In particular, the $(e)$-component of $L_G(F_\mu)$ is non-zero. But $F_1$ has no fixed orbits apart from the trivial ones, in particular, it has no fixed orbit of type $(e)$.


\newpage
\begin{thebibliography}{CMP}

\bibitem{arkowitz}M. Arkowitz and R. Brown, {\it The Lefschetz-Hopf Theorem and axioms for the Lefschetz number}, J. Fixed Point Theory Appl. {\bf 1}, (2004).

\bibitem{ynrohc}B. Chorny, {\it Equivariant cellular homology and its applications}, in {\it High-Dimensional Manifold Topology}, World Scientific, 2003.

\bibitem{dzedzej}Z. Dzedzej, {\it Fixed orbit index for equivariant maps}, Nonlinear Analysis {\bf 47}, (2001).

\bibitem{goncalves}D. Gon\c{c}alves and J. Weber, {\it Axioms for the equivariant Lefschetz number and the equivariant Reidemeister trace}, J. Fixed Point Theory Appl. {\bf 2}, (2007).

\bibitem{hattori}A. Hattori, {\it Rank element of a projective module}, Nagoya Math. J. 25, (1965).

\bibitem{laitinen}E. Laitinen and W. L\"{u}ck, {\it Equivariant Lefschetz classes}, Osaka J. Math. {\bf 26}, (1989).

\bibitem{rosenberg}W. L\"{u}ck and J. Rosenberg, {\it The equivariant Lefschetz fixed point theorem for cocompact proper Lie group actions}, in {\it High-Dimensional Manifold Topology},
World Scientific, 2003.

\bibitem{murayama}M. Murayama, {\it On G-ANRs and their homotopy type}, Osaka J. Math. {\bf 20}, (1983).

\bibitem{nussbaum}R. Nussbaum, {\it Generalizing the fixed point index}, Math. Ann. {\bf 228}, (1977).

\bibitem{tomdieck}T. tom Dieck, {Transformation Groups}, de Gruyter, 1985.

\bibitem{willson}S.-J. Willson, {\it Equivariant homology theories on G-complexes}, Trans. AMS 212, (1975).

\end{thebibliography}
\end{document}